\documentclass[11pt,reqno]{amsart}
\usepackage{latexsym}
\RequirePackage{amsmath}   \RequirePackage{amssymb}
\RequirePackage{amsthm}   \RequirePackage{epsfig}
\theoremstyle{plain}
\newtheorem{theorem}{Theorem}

\newtheorem{lemma}[theorem]{Lemma}
\newtheorem*{corollary*}{Corollary}

\newtheorem*{conjecture*}{Conjecture}

\theoremstyle{definition}

\theoremstyle{remark}

\newcommand{\SC}{{\mathbb C}}  \newcommand{\SD}{{\mathbb D}}  \newcommand{\SN}{{\mathbb N}}
\newcommand {\SR}{{\mathbb R}}    \newcommand{\SZ}{{\mathbb Z}}
    
  \newcommand{\ve}{\varepsilon}  
\def\t{\theta}  \newcommand{\vt}{\vartheta}  
  \newcommand{\vp}{\varphi}

\newcommand{\be}{\begin{equation}}
\newcommand{\ee}{\end{equation}}
\newcommand{\bea}{\begin{eqnarray}}
\newcommand{\eea}{\end{eqnarray}}

\begin{document}

\title{On the failure of Bombieri's conjecture for univalent functions}
\author{Iason Efraimidis}

\address{Departamento de Matem\'aticas, Universidad Aut\'onoma de Madrid, 28049 Madrid, Spain.}
\address{Facultad de Matem\'aticas, Pontificia Universidad Cat\'olica de Chile, Santiago, Chile.}  
\email{iason.efraimidis@mat.uc.cl}
 
\maketitle

\begin{abstract} 
A conjecture of Bombieri \cite{Bom} states that the coefficients of a normalized univalent function $f$ should satisfy
$$
\liminf_{f\to K} \frac{n-{\rm Re\,}a_n}{m-{\rm Re\,}a_m} = \min_{t\in\SR} \, \frac{n\sin t -\sin(nt)}{m\sin t -\sin(mt)}, 
$$
when $f$ approaches the Koebe function $K(z)=\frac{z}{(1-z)^2}$. Recently, Leung \cite{Le} disproved this conjecture for $n=2$ and for all $m\geq3$ and, also, for $n=3$ and for all odd $m\geq5$. Complementing his work we disprove it for all $m>n\geq2$ which are simultaneously odd or even and, also, for the case when $m$ is odd, $n$ is even and $n\leq \frac{m+1}{2}$. We mostly make use of trigonometry, but also employ Dieudonn\'e's criterion for the univalence of polynomials. 
\end{abstract}

\section{Introduction}
Let $S$ denote the class of analytic functions 
$$
f(z) = z +a_2z^2+a_3z^3 +\ldots +a_nz^n +\ldots
$$
which are univalent in the unit disk $\SD = \{z\in\SC : |z|<1 \}$. Throughout the long history of this class one of the motivating forces has been the Bieberbach conjecture, now de Branges' Theorem \cite{dB}, which states that $|a_n|\leq n$ and that the only extremal function is the Koebe function 
$$
K(z) \, = \, \frac{z}{(1-z)^2} \, = \, \sum_{n=1}^\infty n z^n
$$
and its rotations. 

Long before the final solution by de Branges, efforts of many mathematicians culminated in the proof of the local Bieberbach conjecture in an article of Bombieri \cite{Bom}. This weaker conjecture states that $|a_n|\leq n$ for functions in $S$ in a neighborhood of the Koebe function. In the same article, Bombieri conjectured that the numbers 
\be \label{sigma}
\sigma_{mn} = \liminf_{f\to K} \frac{n-{\rm Re\,}a_n}{m-{\rm Re\,}a_m},
\ee
usually referred to as the \emph{Bombieri numbers}, should coincide with the \emph{trigonometric numbers} 
$$
B_{mn} = \min_{t\in\SR} \, \frac{n\sin t -\sin(nt)}{m\sin t -\sin(mt)},
$$
for all $m,n\geq2$. We note that the lower limit in \eqref{sigma} refers to functions $f$ in the class $S$ approaching the Koebe function uniformly on compacta. 

In \cite{PR}, Prokhorov and Roth showed that $\sigma_{mn} \leq B_{mn}$. Also, the local maximum property of the Koebe function yields that $\sigma_{mn}\geq 0$. Setting 
\be \label{A}
A_n(t) = n- \frac{\sin(nt)}{\sin t}, \quad t\in\SR, \;\, n\in\SN,
\ee
it is relatively simple to see that $B_{mn} = 0$ when $m$ is even and $n$ is odd, since in that case $A_n(\pi) = 0 < A_m(\pi)$. Hence $\sigma_{mn} = B_{mn} = 0$ and Bombieri's conjecture is correct when $m$ is even and $n$ is odd. Also, Bshouty and Hengartner \cite{BH} showed that the conjecture is true for analytic variations of the Koebe function and for functions with real coefficients (a simpler proof of the latter appeared in \cite{PR}). Some related results are given in the recent article \cite{AB}. 

The Bombieri conjecture was first disproved by Greiner and Roth \cite{GR} in the case $(m,n) = (3,2)$. They explicitly computed 
$$
\sigma_{32} = \frac{e-1}{4e} < \frac{1}{4} = B_{32}. 
$$ 
Proofs (disproving the conjecture) for the points $(2,4), (3,4)$ and $(4,2)$ were then furnished by Prokhorov and Vasil'ev \cite{PV}, who computed (approximately) the corresponding Bombieri numbers.

Recently, Leung \cite{Le} developed a variational method which allowed him to show that $\sigma_{m2} < B_{m2}$ for all $m\geq3$ and that $\sigma_{m3} < B_{m3}$ for all odd $m\geq5$. He used the \emph{linear} version of Loewner's differential equation
\be \label{Loew}
\frac{\partial f}{\partial t} = z\frac{\partial f}{\partial z} \frac{1+\kappa(t)z}{1-\kappa(t)z},
\ee
whose solutions are \emph{chains} of univalent functions $f(z,t)=e^t (z+a_2(t)z^2 +\ldots), \linebreak t\geq 0$. Any one-slit function in $S$ can be seen as the initial value $f(z)=f(z ,0)$ of such a solution (see \cite{P75}). The \emph{drive function} $\kappa$ has the form $\kappa(t) = e^{i\vt(t)}$, with $\vt$ being real-valued and piecewise continuous on $[0,\infty)$. In the special case when $\kappa\equiv-1$ we get the chain $f(z,t)=e^tK(z)$. Setting $\kappa(t) = -e^{i\ve \vt(t)}$, for $\ve>0$ and some admissible $\vt$ and letting $t=0$, Leung obtained from \eqref{Loew} a variation of Koebe's function, given by
\be \label{var}
f(z)= K(z) +\ve v(z) +\ve^2 q(z) +O(\ve^3),
\ee
for some analytic functions $v$ and $q$ which depend only on the choice of $\vt$. This way Leung re-derived in a simpler fashion the exact same second variation $q$ as Bombieri, who used the \emph{non-linear} version of Loewner's equation. Thus Bombieri's formula (4.1) in \cite{Bom} was obtained by Leung as formula (2.17) in \cite{Le}.

In terms of the coefficients, formula \eqref{var} yields 
$$
a_n = n +\ve v_n +\ve^2 q_n + O(\ve^3).
$$
It is an innate property of the method that the coefficients $v_n$ are purely imaginary and $q_n$ are real. Therefore, 
$$
n-{\rm Re\,}a_n = -\ve^2 q_n + O(\ve^3).
$$
Leung's choice of $\vt$ yields  
\be \label{q_n}
q_n = -\frac{4}{9}(n-1)(2n^2 -4n+3).
\ee
(For the convenience of the reader we have included at the end of the article an appendix where it is shown how, beginning from Bombieri's second variational formula, one can arrive at this number $q_n$.) Hence, 
$$
\sigma_{mn} \leq \lim_{\ve\to0^+} \frac{-\ve^2 q_n + O(\ve^3)}{-\ve^2 q_m + O(\ve^3)} = \frac{q_n}{q_m},
$$
for all $m,n\geq2$. Note that 
$$
\frac{q_n}{q_m} =  \frac{(n-1)(2n^2-4n+3)}{(m-1)(2m^2-4m+3)} <  \frac{n^3-n}{m^3-m} 
$$
for all $m > n\geq2$ since 
$$
\vp(n) = \frac{2n^2-4n+3}{n(n+1)}
$$
increases. Indeed, 
$$
\vp'(x) = \frac{3(2x^2-2x-1)}{x^2(x+1)^2} > 0, \qquad \text{for} \quad x>\frac{1+\sqrt3}{2}\approx 1,366.
$$
Therefore, to disprove Bombieri's conjecture for some $m > n\geq2$, it suffices to show that 
\be \label{B}
B_{mn} = \frac{n^3-n}{m^3-m}. 
\ee
Leung showed that formula \eqref{B} holds true for $n=2$ and for all $m\geq3$ and, also, for $n=3$ and for all odd $m\geq5$. Here, it is our purpose to prove \eqref{B} in some other cases, including the ones just mentioned. In particular, we will prove the following theorem. 

\begin{theorem} \label{Bmn}
Let $m>n\geq2$ be integers such that either \\ 
(a) both $m$ and $n$ are odd, or \\ 
(b) both $m$ and $n$ are even, or \\ 
(c) $m$ is odd, $n$ is even and $n\leq \frac{m+1}{2}$.\\
Then \eqref{B} is true.
\end{theorem}

We have already observed that one can deduce the following corollary. 

\begin{corollary*}
Let $m>n\geq2$ be integers such that either (a), (b) or (c) in \linebreak Theorem \ref{Bmn} holds. Then Bombieri's conjecture for this pair of integers is false. 
\end{corollary*}

Theorem \ref{Bmn} will be proved mainly with the use of trigonometry, but also, in the case when the hypothesis (c) holds, we will employ Dieudonn\'e's criterion for univalent polynomials. 

After carefully examining the relevant graphs for $2\leq n \leq 80$ using the www.desmos.com/calculator software, one is lead to believe that the hypothesis (c) in Theorem \ref{Bmn} can be notably weakened in that the point $(m,n)$ has to be below the straight line that joins the points $(7,6)$ and $(17,14)$. Thus, the following proposition should be true. 

\begin{conjecture*}
If $m>n\geq2$ are integers such that $m$ is odd, $n$ is even and $n< \frac{4m+2}{5}$ then \eqref{B} is true.
\end{conjecture*}

\section{Auxiliary lemmas}

We first mention a criterion for the univalence of polynomials found by Dieudonn\'e \cite{D31} (see also \cite[p.75]{Du2}).

\begin{lemma}[Dieudonn\'e's criterion] \label{Dieu}
The polynomial $p(z) = z+a_2z^2 + \ldots + a_nz^n$ is univalent in $\SD$ if and only if its associated polynomials
$$
q(z;t) \, = \, 1 + a_2 \frac{\sin(2t)}{\sin t} z +\ldots + a_n \frac{\sin(nt)}{\sin t} z^{n-1}
$$
have no zeros in $\SD$ for any choice of the parameter $t\in [0,\pi]$.
\end{lemma}

\vskip0.2cm
We now prove a simple lemma for $A_n(t) = n- \frac{\sin(nt)}{\sin t}$, which we defined in \eqref{A}.
\begin{lemma} \label{lemA}
For all $t\in\SR$ and $n\geq 2$, we have
$$
A_n(t) \geq 0 \qquad \text{and} \qquad A_n(2\pi-t) = A_n(t).
$$
Also, $A_n$ vanishes only for $t=2\ell\pi, \ell\in\SZ$, when $n$ is even and only for $t=\ell\pi$, $\ell\in\SZ$, when $n$ is odd. 
\end{lemma}
\begin{proof}
The symmetry is fairly obvious. Due to it we may restrict our attention to the interval $[0,\pi]$. Using L'Hospital's rule we find that 
$$
A_{2k}(0) = A_{2k+1}(0) = A_{2k+1}(\pi) = 0, \qquad  A_{2k}(\pi) = 4k,
$$
for any $k\geq1$. Now, for $t\in(0,\pi)$, $A_n(t) >0$ is equivalent to 
$$
\vp(t) := n\sin t -\sin(nt) >0,  
$$
whose derivative is
$$
\vp'(t) = n(\cos t - \cos(nt)).
$$
If $t_0$ is a critical point of $\vp$ then $\sin t_0 = \pm \sin(nt_0)$. Hence 
$$
\vp(t_0) = (n\mp1) \sin t_0 >0  
$$
and the proof is complete. 
\end{proof}

\vskip0.2cm
We wish to remark that there are at least two more ways to prove this lemma. First, we could apply Dieudonn\'e's criterion to the univalent polynomial $z-z^n/n$ (which is, moreover, starlike \cite[Thm.~2.3]{Br}) and let $z\to 1$ along the real axis. Alternatively, for odd $n$ we could use the connection with the Dirichlet kernel 
$$
D_n(x) = \frac{\sin(n+1/2)x}{\sin x/2} = 1+2\sum_{j=1}^n \cos(jx),
$$
which is $A_{2k+1}(t) = 2k+1-D_k(2t)$ (see \cite[\S 8.4]{Du4}, for example). For even $n$ we would simply have to adjust the proof of the above expansion in cosines, where the trick with telescoping sums works equally well. However, we note that only the latter of these two proofs yields naturally the strict inequality in the open interval $(0,\pi)$. 

\begin{lemma} \label{lem+2}
For all integers $n\geq 2$ and for all $t\in (0,\pi)$ it holds that 
\be \label{ind}
\frac{A_n(t)}{n^3-n} \geq \frac{A_{n+2}(t)}{(n+2)^3-(n+2)}.
\ee
\end{lemma}
\vskip0.2cm
\begin{proof}
We set $N=n+1\geq3$ and see that \eqref{ind} is equivalent to 
$$
N(N+1)(N+2)A_{N-1}(t) \, \geq \, N(N-1)(N-2)A_{N+1}(t),
$$
which, in turn, is equivalent to
$$
4(N^2-1) -(N+1)(N+2)\frac{\sin(N-1)t}{\sin t} +(N-1)(N-2)\frac{\sin(N+1)t}{\sin t} \geq 0.
$$
Multiplying by $\frac{1}{2}\sin t$, expanding the sines of the sums and setting
\be \label{f}
\Phi(t) \, = \, 2(N^2-1)\sin t -3N \sin(Nt) \cos t +(N^2+2) \cos(Nt) \sin t,
\ee
we see that the above is equivalent to $\Phi(t)\geq 0$. We note that 
$$
\Phi\left(\frac{\pi}{2}\right) = 2N^2-2 +(N^2+2) \cos\left(\frac{N\pi}{2}\right) \geq N^2-4 >0,
$$
for shortly we will need to consider $t\neq \frac{\pi}{2}$. We compute
\begin{align} \label{f'}
\frac{\Phi'(t)}{N^2-1} \, & = \, 2\cos t -2\cos(Nt) \cos t -N \sin(Nt) \sin t \nonumber \\
& = \, 2 \sin\left(\frac{Nt}{2}\right) \left( 2 \sin\left(\frac{Nt}{2}\right) \cos t  -N \cos\left(\frac{Nt}{2}\right) \sin t  \right). 
\end{align}
Hence, one set of the roots of $\Phi'$ comes from $\sin\left(\frac{Nt}{2}\right) = 0$. Solutions of this equation satisfy $Nt_k = 2k\pi, k\in\SZ$, and it is easy to check that
$$
\Phi(t_k) = 3N^2 \sin t_k > 0.
$$
The rest of the roots of $\Phi'$ comes from 
\be \label{tan}
\tan\left(\frac{Nt}{2}\right) = \frac{N}{2} \tan t,
\ee
if we momentarily consider that $\cos\left(\frac{Nt}{2}\right) \neq 0$. We return to \eqref{f} and compute 
$$
\Phi(t) \, = \, (N^2-4)\sin t + 2 \cos^2\left(\frac{Nt}{2}\right) \sin t \left( N^2+2 -3N\frac{\tan\left(\frac{Nt}{2}\right)}{\tan t} \right).
$$
Hence, if $t^*$ satisfies \eqref{tan} then 
$$
\Phi(t^*) \, = \, (N^2-4)\sin t^* \left( 1-  \cos^2\left(\frac{Nt^*}{2}\right) \right) \geq 0, 
$$
which was our goal. Therefore, it is only left to consider the case when $\cos\left(\frac{Nt}{2}\right) = 0$ for some critical point of $\Phi$. But this would give $Nt=(2k+1)\pi,k\in\SZ$ and a substitution in \eqref{f'} yields
$$
\frac{\Phi'(t)}{N^2-1} = 4 \cos t,
$$
which vanishes only at $t=\frac{\pi}{2}$, a point we have previously considered. 
\end{proof}

\vskip0.3cm
\section{Proof of Theorem \ref{Bmn}}
We now proceed with the proof of our main theorem. 

\begin{proof}[Proof of Theorem \ref{Bmn}]
We set 
$$
\vp_{mn}(t) := \frac{n\sin t -\sin(nt)}{m\sin t -\sin(mt)} = \frac{A_n(t)}{A_m(t)}, \quad t\in[0,2\pi],
$$
whose minimum is the number $B_{mn}$. In view of the symmetry of $A_n$ (stated in Lemma~\ref{lemA}) we may restrict our attention to $t$ in $[0,\pi]$.

\vskip0.1cm
Suppose first that either the hypothesis (a) or (b) holds, that is, $m$ and $n$ are simultaneously odd or even. Note that 
$$
\vp_{mn}(0) = \vp_{mn}(\pi) = \frac{n^3-n}{m^3-m} \qquad \text{for odd} \;\; m,n
$$
and that 
$$
\vp_{mn}(0) = \frac{n^3-n}{m^3-m} < \frac{n}{m} = \vp_{mn}(\pi) \qquad \text{for even} \;\; m,n.
$$
Hence, our goal is to show that 
$$
\frac{A_n(t)}{A_m(t)} \geq \frac{n^3-n}{m^3-m} \qquad \text{for} \;\; t\in (0,\pi).
$$
But this follows directly from Lemma \ref{lem+2} after a finite number of iterations 
$$
\frac{A_n(t)}{n^3-n} \geq \frac{A_{n+2}(t)}{(n+2)^3-(n+2)} \geq \frac{A_{n+4}(t)}{(n+4)^3-(n+4)} \geq \cdots \geq \frac{A_m(t)}{m^3-m}.
$$

\vskip0.1cm
Suppose now that the hypothesis (c) holds, that is, $m$ is odd, $n$ is even and $n\leq \frac{m+1}{2}$. Note that 
$$
\vp_{mn}(0) = \frac{n^3-n}{m^3-m} < +\infty = \vp_{mn}(\pi).
$$
Once again, in view of Lemma \ref{lem+2} it suffices to prove that 
$$
\frac{A_n(t)}{A_{m_0}(t)} \geq \frac{n^3-n}{m_0^3-m_0} \qquad \text{for} \;\; t\in (0,\pi),
$$
where $m_0=2n-1$. This is equivalent to 
$$
4(2n-1)A_n(t) \geq (n+1)A_{2n-1}(t),
$$
which, in turn, is the same as 
\be \label{pos}
1- \frac{4}{3n-1} \frac{\sin(nt)}{\sin t} +\frac{n+1}{(2n-1)(3n-1)}\frac{\sin\big((2n-1)t\big)}{\sin t} \geq 0. 
\ee
It would clearly suffice to prove that 
\be \label{n-neq}
1- \frac{4}{3n-1} \frac{\sin(nt)}{\sin t} z^{n-1} +\frac{n+1}{(2n-1)(3n-1)}\frac{\sin\big((2n-1)t\big)}{\sin t}z^{2n-2}  \neq 0, 
\ee
for all $z\in\SD$, since this would imply that for $z=x\in [0,1)$ the function in \eqref{n-neq} is positive and \eqref{pos} would follow after letting $x\to 1^-$. In view of Dieudonn\'e's criterion (Lemma \ref{Dieu}), \eqref{n-neq} is equivalent to the statement that the function
\be \label{Bran}
f(z) = z -\frac{4}{3n-1} z^n +\frac{n+1}{(2n-1)(3n-1)} z^{2n-1}
\ee
belongs to the class $S$. We will actually prove more: we will show that $f$ is starlike, which means that $f$ is univalent and that for every $w\in f(\SD)$ the line segment $[0,w]$ lies entirely in $f(\SD)$.

First, we see that the roots of 
$$
\frac{f(z)}{z} = 1 -\frac{4}{3n-1} z^{n-1} +\frac{n+1}{(2n-1)(3n-1)} z^{2n-2}
$$
satisfy 
$$
z^{n-1} = \frac{2(2n-1)\pm i(n-1)\sqrt{3(2n-1)}}{n+1},
$$
and therefore
$$
|z|^{2n-2} = \frac{(2n-1)(3n^2+2n-1)}{(n+1)^2} > 1.
$$
This shows that the function 
$$
p(z) = \frac{zf'(z)}{f(z)}
$$ 
is analytic in $\overline{\SD}$ and so in order to apply the well-known criterion for starlikeness \cite[\S 2.5]{Du2} it suffices to show that 
\be \label{p}
{\rm Re\,} p(z) \geq 0, \quad \text{for} \;\; \; |z|=1.
\ee
We compute 
$$
\frac{p(z)}{2n-1} \, = \, \frac{(n+1)z^{2n-2} -4nz^{n-1}+3n-1}{(n+1)z^{2n-2} -4(2n-1)z^{n-1}+(2n-1)(3n-1)}
$$
and let $z^{n-1} = e^{i\t}, \, \t\in\SR$. We then have
\begin{align*}
\frac{p(z)}{2n-1} \, & = \, \frac{(n+1) e^{i\t} -4n +(3n-1)e^{-i\t}}{(n+1) e^{i\t} -4(2n-1) +(2n-1)(3n-1)e^{-i\t}} \\
& = \, \frac{2n(\cos\t-1)-(n-1)i\sin\t}{(3n^2-2n+1)\cos\t-2(2n-1)-3n(n-1)i\sin\t}.
\end{align*}
Multiplying by the complex conjugate of the denominator we see that \eqref{p} is equivalent to 
\begin{align*}
0\leq & \, 2n (\cos\t-1) \big[(3n^2-2n+1)\cos\t-2(2n-1)\big] +3n(n-1)^2 \sin^2\t \\
= & \, n (n+1) (3n-1) (\cos\t-1)^2,
\end{align*}
which is true. The proof is complete.  
\end{proof}

Note that the polynomial \eqref{Bran} resembles the polynomials considered in a theorem of Brannan \cite[Thm.~3.1]{Br}, which gave necessary and sufficient conditions for a polynomial of the form 
$$
z \,+ \, a \, z^{n} + \frac{z^{2n-1}}{2n-1},  \quad a\in\SC,
$$
to be univalent. Even though this theorem can not be applied here, the main ingredient in its proof, which is the \emph{Cohn rule} (see \cite[Lem.~1.2]{Br}), could be directly applied to prove \eqref{n-neq} and thus give an alternative ending of the proof of Theorem~\ref{Bmn}.

\section{Appendix: Calculation of $q_n$}

Here our starting point will be Bombieri's formula (4.1) in \cite{Bom}. According to it, if $\phi$ is a function in $L^2[0,1]$ then a second variation of the Koebe function is given by $q(z) = Q\big(K(z)\big)$, where
\be \label{Bom-2nd}
Q(w) \, = \, - w^2 \int_0^1 \frac{\phi(u)^2}{U} du \, - 2w^3 \int_0^1  \int_0^u \left(3+\frac{1}{V}\right) \frac{\phi(u)\phi(v)}{\sqrt{UV}} dv du,
\ee
$U=1+4uw$ and $V=1+4vw$. Note the following homogeneity property: if we replace $\phi$ by $c \, \phi$ ($c\in\SR$) then instead of $Q$ we obtain $c^2Q$. In fact, our aim here is to show how a specific choice of $\phi$ yields 
$$
q_n = -\frac{1}{9}(n-1)(2n^2 -4n+3),
$$
which is a scalar multiple of \eqref{q_n}. We will provide a slightly more direct approach than Leung who, for additional purposes, considers \eqref{Bom-2nd} with variable $z\in\SD$ and integration over the interval $[-1,1]$ in order to use properties of classical orthogonal polynomials. 

We rewrite \eqref{Bom-2nd} as 
\begin{align*}
Q(w)   \, = \, & - w^2 \int_0^1 \frac{\phi(u)^2}{1+4uw} du \\
& -6 w^3 \int_0^1  \int_0^u  \frac{\phi(u)\phi(v)}{\sqrt{1+4uw}\sqrt{1+4vw}}dv du \\
& -2 w^3 \int_0^1  \int_0^u  \frac{\phi(u)\phi(v)}{\sqrt{1+4uw}(1+4vw)^{3/2}}dv du
\end{align*}
and denote by $I_1, I_2$ and $I_3$ the three integrals in the order appearance, so that 
$$
Q(w) \, = \, -w^2 (I_1 + 6w I_2 + 2w I_3).
$$
We observe that the integrand in $I_2$ is symmetric in $u$ and $v$ and therefore its integral over the lower triangle of $[0,1]^2$ (which is $I_2$) is equal to the integral over the upper triangle. Hence 
$$
I_2 \, = \, \frac{1}{2} \left( \int_0^1 \frac{\phi(u)}{\sqrt{1+4uw} }du \right)^2. 
$$
To deal with $I_3$ we note that 
$$
\frac{2w}{(1+4vw)^{3/2} } \, = \, - \frac{\partial}{\partial v }  \left( \frac{1}{\sqrt{1+4vw}} \right).
$$
An integration by parts now yields
\begin{align*}
2w I_3  \,= &\, - \int_0^1 \frac{\phi(u)^2}{1+4uw} du + \phi(0) \int_0^1 \frac{\phi(u)}{\sqrt{1+4uw} }du \\
& + \int_0^1  \int_0^u  \frac{\phi(u)\phi'(v)}{\sqrt{1+4uw}\sqrt{1+4vw}}dv du. 
\end{align*}
In total, we have
\begin{align} \label{Q(w)}
Q(w)   \, = \, & -w^2 \phi(0) \int_0^1 \frac{\phi(u)}{\sqrt{1+4uw} }du \, -3w^3 \left( \int_0^1 \frac{\phi(u)}{\sqrt{1+4uw} }du \right)^2 \\ 
& - w^2 \int_0^1  \int_0^u  \frac{\phi(u)\phi'(v)}{\sqrt{1+4uw}\sqrt{1+4vw}}dv du. \nonumber
\end{align}
We now choose $\phi(u) = 1-u$. It is helpful to compute
$$
\int_0^u  \frac{dv}{\sqrt{1+4vw}}  =  \frac{\sqrt{1+4uw}-1}{2w}
$$
and (integrating by parts): 
$$
\int_0^1 \frac{u \;du }{\sqrt{1+4uw}} =  \frac{\sqrt{1+4w}}{2w} - \frac{(1+4w)^{3/2}-1}{12w^2}.
$$
Then we can compute the integrals in \eqref{Q(w)}. They are
$$
\int_0^1 \frac{\phi(u)}{\sqrt{1+4uw} }du = \frac{(1+4w)^{3/2}-6w-1}{12w^2}
$$
and 
$$
\int_0^1  \int_0^u  \frac{\phi(u)\phi'(v)}{\sqrt{1+4uw}\sqrt{1+4vw}}dv du = \frac{(1+4w)^{3/2}-6w^2-6w-1}{24w^3}.
$$
We substitute these in \eqref{Q(w)} and after elementary but cumbersome calculations we obtain
$$
Q(w)  = \frac{1+4w}{6} \left( \sqrt{1+4w}  -1 -2w \right). 
$$
Setting $w=K(z) = \frac{z}{(1-z)^2}$ we get
$$
q(z) = Q\big(K(z)\big) = -\frac{z^2(1+z)^2}{3(1-z)^4}.
$$
Finally, we compute the $n$-th coefficient of $q$ with the aid of the standard formula 
$$
\frac{1}{(1-z)^4} = \sum_{n=0}^\infty \frac{(n+1)(n+2)(n+3)}{6} z^n.
$$ 

\vspace{0,2cm}
\emph{Acknowledgements}. The author has been supported by a fellowship of the International Excellence Graduate Program in Mathematics at Universidad Aut\'onoma de Madrid (422Q101) and also partially supported by grant MTM2015-65792-P by MINECO/FEDER-EU. This work forms part of his Ph.D. thesis at UAM under the supervision of professor Dragan Vukoti\'c. The author would like to thank him for his encouragement and help. 

The author would also like to thank professor Yuk J. Leung for providing him with a copy of \cite{Le} and suggesting that formula \eqref{B} should be true under the hypothesis (a) and (b) of Theorem \ref{Bmn}.

\end{document}